\documentclass{article}

\usepackage{graphicx,amsmath,amsthm,amssymb,subfig,float,mathrsfs,mathabx}
\usepackage[margin=1in, paperwidth=21.6cm, paperheight=28cm]{geometry}
\theoremstyle{definition}

\newtheorem{theorem}{Theorem}[section]
\newtheorem{lemma}[theorem]{Lemma}

\numberwithin{equation}{section}

\newtheorem{proposition}[theorem]{Proposition}
\newtheorem{corollary}[theorem]{Corollary}
\newtheorem{example}[theorem]{Example}
\newtheorem{definition}[theorem]{Definition}

\theoremstyle{definition}

\title{Approximate Hamiltonian Symmetry Groups and Recursion Operators for Perturbed Evolution Equations}

\author{
M. Nadjafikhah
\\ \small School of Mathematics, Iran University of Science and Technology,\\ \small Narmak, Tehran, 1684613114, I.R.Iran. e-mail : m\_nadjafikhah@iust.ac.ir
\\  A. Mokhtary
\\ \small Department of Mathematics, Payame Noor University, \\ \small Lashkarak, Tehran,19395-4697, I.R.Iran. e-mail : mokhtary@phd.pnu.ac.ir
}
\begin{document}
\maketitle
\begin{abstract}
In this paper, the method of approximate transformation groups which was proposed by Baikov, Gazizov
and Ibragimov \cite{Baikov1,Baikov2}, is extended on Hamiltonian and bi-Hamiltonian systems of evolution equations.
Indeed, as a main consequence, this extended procedure is applied in order to compute the
approximate conservation laws and approximate recursion operators corresponding to these type of equations.
In particular, as an application, a comprehensive analysis of the problem
of approximate conservation laws and approximate recursion operators associated to the
Gardner equation with the small parameters is presented.
\end{abstract}

\hspace{9pt}\textit{Keywords :} Perturbed evolution equation, Approximate Hamiltonian symmetry group,  Approximate recursion operators
\vspace{.2cm}

\hspace{9pt}2010 Mathematics Subject Classification (MSC2010): 76M60, 35B20, 35Q35
\section{Introduction}
The investigation of the exact solutions of nonlinear evolution equations has a
fundamental role in the nonlinear physical phenomena. One of the significant and
systematic methods for obtaining special solutions of systems of nonlinear
differential equations is the classical symmetries method, also
called Lie group analysis. This well known approach originated at the end of nineteenth century
from the pioneering work of Lie \cite{Lie1}. The fact that symmetry reductions for many PDEs
can not be determined via the classical symmetry method, motivated the creation of several
generalizations of the classical Lie group approach for symmetry reductions. Consequently,
several alternative reduction methods have been proposed, going beyond Lie's classical
procedure and providing further solutions. One of these techniques which is extremely applied
particularly for nonlinear problems is perturbation analysis. It is worth mentioning that sometimes differential equations
which appear in mathematical modelings are presented with terms involving a parameter called the perturbed term.
Because of the instability of the Lie point symmetries with respect to perturbation of coefficients of
differential equations, a new class of symmetries has been created for such equations, which are known as approximate (perturbed)
symmetries.
In the last century, in order to have
the utmost result from the methods, combination of Lie symmetry method and perturbations are investigated
 and two different so called \textit{approximate symmetry methods}(ASM) have been developed.
The first method is due to Baikov, Gazizov and Ibragimov \cite{Baikov1,Baikov2}. The second procedure was
proposed by Fushchich and Shtelen \cite{Fush} and
later followed by Euler et al \cite{Euler1,Euler2}. This method is generally based on the
perturbation of dependent variables. In \cite{Pak,Wilst1}, a comprehensive
comparison of these two methods is presented.\\
As it is well known, Hamiltonian systems of differential equations are one of the famous and
significant concepts in physics. These important systems appear in the various fields of physics such as
motion of rigid bodies, celestial mechanics, quantization theory, fluid mechanics, plasma physics, etc.
Due to the significance of Hamiltonian structures, in this paper, by applying the
linear behavior of the  Euler operator, characteristics, prolongation and Fr\'{e}chet derivative of vector fields, we have extended ASM on the
Hamiltonian and bi-Hamiltonian systems of evolution equations, in order to investigate the interplay between approximate symmetry
groups, approximate conservation laws and approximate recursion operators.\\
The structure of the present paper is as follows: In section 2, some necessary preliminaries
regarding to the Hamiltonian structures are presented. In section 3, a comprehensive investigation of the
approximate Hamiltonian symmetries and approximate conservation laws associated to the perturbed evolution equations
is proposed. Also, as an application of this procedure, approximate Hamiltonian symmetry
groups, approximate bi-Hamiltonian Structures and  approximate conservation laws of the Gardner equation
are computed. In section 4, the approximate recursion operators are studied and the proposed technique is implemented
for the Gardner equation as an application. Finally, some concluding remarks are mentioned at the end of the paper.
\section{Preliminaries}
In this section, we will mention some necessary preliminaries regarding to
Hamiltonian structures. In order to be familiar with the general concepts of the ASM, refer
to \cite{Ibra}. It is also worth mentioning that most of this paper's
definitions, theorems and techniques regarding to  Hamiltonian and bi-Hamiltonian structures are inspired
from \cite{Olver1}.\\
Let $M \subset X \times U $  denote a fixed connected open subset of the space of
independent and dependent variables $x=(x^1,\cdots,x^p)$ and $u=(u^1,\cdots,u^q)$.
The algebra of differential functions $P(x,u^{(n)})=p[u]$ over $M$
is denoted by $\mathscr{A}$. We further
define $\mathscr{A}^l$ to be the vector space of $l$-tuples of differential functions, $P[u] =
(P_1[u],\ldots,P_l[u])$, where each $P_j \in \mathscr{A}$.\\
A \textit{generalized vector field} will be a (formal) expression of the
form
\begin{equation}
\label{eq:eq12}
\mathbf{v}=\sum_{i=1}^p\xi^i[u]\frac{\partial}{\partial x^i}+\sum_{\alpha=1}^q\phi_\alpha[u]\frac{\partial}{\partial u^\alpha}
\end{equation}
in which $\xi^i$ and $\phi_\alpha$ are smooth differential functions. The  \textit{Prolonged} generalized vector field
can be defined as follows:
\[
\mathbf{pr^{(n)}}\,\mathbf{v}=\mathbf{v}+\sum_{\alpha=1}^q\,\sum_{\sharp J\leq n}\,\phi_\alpha^J[u]\frac{\partial}{\partial u_J^\alpha},
\]
whose coefficients are determined by the formula
\begin{equation}
\label{eq:eq13}
\phi_\alpha^J=D_J\Big(\phi_\alpha-\sum_{i=1}^p\xi^iu_i^\alpha \Big)+\sum_{i=1}^p\xi^iu_{J,i}^\alpha ,
\end{equation}
with the same notation as before. Given a generalized vector field $\mathbf{v}$, its \textit{infinite prolongation} (or briefly \textit{prolongation}) is the formally infinite sum
\begin{equation}
\label{eq:eq14}
\mathbf{pr\,v}=\sum_{i=1}^p\, \xi^i\frac{\partial}{\partial x^i}+\sum_{\alpha=1}^q\,\sum_J \phi_\alpha^J\frac{\partial}{\partial u_J^\alpha},
\end{equation}
where each $\phi_\alpha^J$ is given by (\ref{eq:eq13}), and the sum in (\ref{eq:eq14}) now extends over
all multi-indices $J=(j_1,\ldots,j_k)$ for $k\geq 0,\, 1\leq j_k\leq p$.

A generalized vector field v is a \textit{generalized infinitesimal
symmetry} of a system of differential equations
\[
\Delta_\nu [u]=\Delta_\nu(x,u^{(n)})=0,\qquad \nu=1,\ldots,l,
\]
if and only if
\begin{equation}
\label{eq:eq15}
\mathbf{pr}\,\mathbf{v}[\Delta_\nu]=0,\qquad \nu=1,\ldots,l,
\end{equation}
for every smooth solution $u=f(x)$.

Among all the generalized vector fields, those in which the coefficients $\xi^i[u]$
of the $\frac{\partial}{\partial x^i}$ are zero play a distinguished role. Let $Q[u]=(Q_1[u],\ldots,Q_q[u])\in \mathscr{A}^q$ be a $q$-tuple of differential
functions. The generalized vector field
\[
\mathbf{v}_Q=\sum_{\alpha=1}^q Q_\alpha [u]\frac{\partial}{\partial u^\alpha}
\]
is called an \textit{evolutionary vector field}, and $Q$ is called its \textit{characteristic}.

A manifold $M$ with a Poisson bracket is called a \textit{Poisson manifold}, the
bracket defining a \textit{Poisson structure} on $M$. Let $M$ be a Poisson manifold and $H:\,M\to\mathbb{R}$  be a smooth
function. The \textit{Hamiltonian vector field} associated with $H$ is the unique smooth
vector field $\mathbf{\hat{v}_H}$ on $M$ satisfying the following identity
\begin{equation}
\label{eq:eq18}
\mathbf{\hat{v}_H}=\{F,H\}=-\{H,F\}
\end{equation}
for every smooth function $F:\,M\to\mathbb{R}$. The equations governing the flow of  $\mathbf{\hat{v}_H}$
are referred to as \textit{Hamilton's equations} for the ``Hamiltonian" function $H$.

Let $x = (x^1,\ldots,x^m)$ be local
coordinates on $M$ and $H(x)$ be a real-valued function. The following basic formula
can be obtained for the Poisson bracket.
\begin{equation}
\label{eq:eq20}
\{F,H\} = \sum_{i=1}^m\sum_{j=1}^m\{x^i,x^j\}\frac{\partial F}{\partial x^i}\frac{\partial H}{\partial x^j}
\end{equation}
In other words, in order to compute the Poisson bracket of
any pair of functions in some given set of local coordinates, it suffices to know
the Poisson brackets between the coordinate functions themselves. These
basic brackets,
\begin{equation}
\label{eq:eq21}
J^{ij}(x)=\{x^i,x^j\},\qquad i,j=1,\ldots,m,
\end{equation}
are called the \textit{structure functions} of the Poisson manifold $M$ relative to the
given local coordinates, and serve to uniquely determine the Poisson
structure itself. For convenience, we assemble the structure functions into a skew-
symmetric $m \times m$ matrix $J(x)$, called the \textit{structure matrix} of $M$. Using $\nabla H$ to
denote the (column) gradient vector for $H$, the local coordinate form  (\ref{eq:eq20}) for
the Poisson bracket can be written as
\begin{equation}
\label{eq:eq22}
\{F,H\} =\nabla F\cdot J\nabla H.
\end{equation}

Therefore, in the given coordinate chart, Hamilton's equations take the
form
\begin{equation}
\label{eq:eq24}
\frac{dx}{dt}=J(x)\nabla H(x).
\end{equation}
Alternatively, using (\ref{eq:eq20}), we could write this in the ``bracket form" as follows:
\[
\frac{dx}{dt}=\{x,H\},
\]
the $i$-th component of the right-hand side being $\{x^i,H\}$. A system of first
order ordinary differential equations is said to be a \textit{Hamiltonian system}
if there is a Hamiltonian function $H(x)$ and a matrix of functions $J(x)$
determining a Poisson bracket (\ref{eq:eq22}) whereby the system takes the form
(\ref{eq:eq24}).

If
\[
\mathscr{D}=\sum_J\,P_J[u]D_J,\qquad P_J\in\mathscr{A}
\]
is a differential operator, its (formal)\textit{adjoint}  is the differential operator $\mathscr{D}^*$
which satisfies
\[
\int_{\Omega}\,P\cdot\mathscr{D}Q\,dx=\int_\Omega\,Q\cdot\mathscr{D}^*P\,dx
\]
for every pair of differential functions $P,Q\in\mathscr{A}$ which vanish when $u\equiv 0$. Also, for every domain $\Omega\subset \mathbf{R}^p$ and every function $u=f(x)$ of compact support in $\Omega$. An operator
$\mathscr{D}
$ is \textit{self-adjoint} if $\mathscr{D}^*=\mathscr{D}$; it is \textit{skew-adjoint} if $\mathscr{D}^*=-\mathscr{D}$.

The principal
innovations needed to convert a Hamiltonian system of ordinary differential
equations (\ref{eq:eq24}) to a Hamiltonian system of evolution equations are as follows (refer to \cite{Olver1} for more details):
\begin{itemize}
\item[(i)] replacing the Hamiltonian function $H(x)$ by a Hamiltonian functional $\mathscr{H}[u]$,
\item[(ii)]replacing the vector gradient operation $\nabla H$ by the variational derivative $\delta\mathscr{H}$ of the Hamiltonian functional, and
\item[(iii)] replacing the skew-symmetric matrix $J(x)$ by a skew-adjoint differential
operator $\mathscr{D}$ which may depend on u.
\end{itemize}
The resulting Hamiltonian system will take the form
\[
\frac{\partial u}{\partial t}=\mathscr{D}\cdot\delta\mathscr{H}[u]
\]
Clearly, for
a candidate Hamiltonian operator $\mathscr{D}$ the correct expression for the
corresponding Poison bracket has the form
\begin{equation}
\label{eq:eq27}
\{\mathscr{P},\mathscr{L}\}=\int\,\delta\mathscr{P}\cdot\mathscr{D}\delta\mathscr{L}\,dx,
\end{equation}
whenever  $\mathscr{P},\mathscr{L}\in\mathscr{F}$ are functionals. Off course, the Hamiltonian operator $\mathscr{D}$
must satisfy certain further restrictions in order that (\ref{eq:eq27}) be a true Poisson
bracket. A linear operator $\mathscr{D}:\mathscr{A}^q\to\mathscr{A}^q$ is called \textit{Hamiltonian} if its
Poisson bracket (\ref{eq:eq27}) satisfies the conditions of \textit{skew-Symmetry}
and the \textit{Jacobi identity}.
\begin{proposition}
Let $\mathscr{D}$ be a Hamiltonian operator with Poisson bracket (\ref{eq:eq27}).
To each functional $\mathscr{H}=\int\,H\,dx\in\mathscr{F}$,  there is an evolutionary vector field $\mathbf{pr\,\hat{v}_\mathscr{H}}$,
called the Hamiltonian vector field associated with $\mathscr{H}$, which for all functionals $\mathscr{P}\in\mathscr{F}$ satisfies the following identity:
\begin{equation}
\label{eq:eq30}
\mathbf{pr\,\hat{v}_\mathscr{H}}(\mathscr{P})=\{\mathscr{P},\mathscr{H}\}
\end{equation}
Indeed, $\mathbf{\hat{v}_\mathscr{H}}$  has characteristic $\mathscr{D}\delta\mathscr{H}=\mathscr{D}\mathbf{E}(H)$, in which $\mathbf{E}$ is Euler operator.
(Proposition 7.2 of \cite{Olver1})
\end{proposition}

\section{Approximate Hamiltonian Symmetries and Approximate\\ Conservation Laws}

Consider a system of perturbed evolution equations
\begin{equation}
\label{eq:eq31}
\frac{\partial u}{\partial t}=P[u,\varepsilon]
\end{equation}
in which $P[u,\varepsilon]=P(x,u^{(n)},\varepsilon)\in \mathscr{A}^q$,
$x\in \mathbb{R}^p$, $u\in\mathbb{R}^q$ and $\varepsilon$ is a parameter.

Substituting according to (\ref{eq:eq31}) and its derivatives, we see that
any evolutionary symmetry must be equivalent to one whose characteristic
$Q[u,\varepsilon]=Q(x,t,u^{(m)},\varepsilon)$ depends only on $x$, $t$, $u$, $\varepsilon$ and the $x$-derivatives of $u$. On
the other hand, (\ref{eq:eq31}) itself can be considered as the equations corresponding to the flow
$\exp(t\mathbf{v}_p)$ of the evolutionary vector field with characteristic $P$.
The symmetry
criterion (\ref{eq:eq15}), which in this case is
\begin{equation}
\label{eq:eq32}
D_t\,Q_\nu=\mathbf{pr\,v}_Q(P_\nu)+o(\varepsilon^p),\qquad\nu=1,\ldots,q,
\end{equation}
can be readily seen to be equivalent to the following Lie bracket condition on
the two approximate generalized vector fields. Indeed, this point generalizes  the correspondence between
symmetries of systems of first order perturbed ordinary differential equations and the
Lie bracket of the corresponding vector fields.

Considering above assumptions, some useful relevant theorems and definitions
could be rewritten as follows:
\begin{proposition} An approximate evolutionary vector field  $\mathbf{v}_Q$ is a symmetry of the system
of perturbed evolution equations $u_t=P[u,\varepsilon]$ if and only if
\begin{equation}
\label{eq:eq33}
\frac{\partial \mathbf{v}_Q}{\partial t}+[\mathbf{v}_P,\mathbf{v}_Q]=o(\varepsilon^p)
\end{equation}
holds identically in $(x,t,u^{(m)},\varepsilon)$. (Here $\partial\mathbf{v}_Q/\partial t$ denotes the evolutionary vector field with characteristic $\partial Q/\partial t$.)
\end{proposition}

Any approximate conservation law of a system of perturbed evolution
equations takes the form
\begin{equation}
\label{eq:eq34}
D_tT+\mathbf{Div}\,X=o(\varepsilon^p),
\end{equation}
in which $\mathbf{Div}$ denotes spatial divergence. Without loss of generality,
the conserved density $T(x,t,u^{(n)},\varepsilon)$ can be assumed
to depend only on $x$-derivatives of
$u$. Equivalently, for $\Omega\subset X$, the functional
\[
\mathscr{T}[t;u,\varepsilon]=\int_\Omega T(x,t,u^{(n)},\varepsilon)\,dx
\]
is a constant, independent of $t$, for all solutions $u$ such that $T(x,t,u^{(n)},\varepsilon)\to 0$ as
$x\to\partial \Omega$.
Note that if $T(x,t,u^{(n)},\varepsilon)$ is any such differential function, and $u$ is a solution
of the perturbed evolutionary system $u_t=P[u,\varepsilon]$, then
\[
D_tT\approx\partial_tT+ \mathbf{pr\,v}_p(T),
\]
where $\partial_t=\partial/\partial t$ denotes the partial $t$-derivative. Thus $T$ is the density for a
conservation law of the system if and only if its associated functional
$\mathscr{T}$ satisfies the following identity
\begin{equation}
\label{eq:eq35}
\partial\mathscr{T}/\partial t+\mathbf{pr\,v}_p(\mathscr{T})=o(\varepsilon^p).
\end{equation}
In the case that our system is of Hamiltonian form, the bracket relation (\ref{eq:eq30})
immediately leads to the Noether relation between approximate Hamiltonian symmetries
and approximate conservation laws.

\begin{definition}
Let $\mathscr{D}$ be a $q\times q$ approximate Hamiltonian differential operator. An
\textit{approximate distinguished functional} for $\mathscr{D}$ is a functional  $\mathscr{G}\in \mathscr{F}$  satisfying $\mathscr{D}\delta \mathscr{G}=o(\varepsilon^p)$  for all $x,u$.
\end{definition}
In other words, the Hamiltonian system corresponding to a distinguished
functional is completely trivial: $u_t = 0$.

Now, according to \cite{Olver1},  the perturbed Hamiltonian version of Noether's theorem can be presented as follows:
\begin{theorem}
\label{T:thm3}
 Let $u_t=\mathscr{D}\delta\mathscr{H}$ be a Hamiltonian system of perturbed evolution equations.
An approximate Hamiltonian vector field $\mathbf{\hat{v}}_\mathscr{P}$ with characteristic $ \mathscr{D}\delta\mathscr{P},\mathscr{P}\in\mathscr{F}$, determines an
approximate generalized symmetry group of the system if and only if there is an equivalent
functional $\tilde{\mathscr{P}}\approx\mathscr{P}-\mathscr{G}$ differing only from $\mathscr{p}$ by a time-dependent  approximate
distinguished functional $\mathscr{G}[t;u,\varepsilon]$, such that $\tilde{\mathscr{P}}$ determines an approximate conservation law.
\end{theorem}

\begin{example}
The Gardner equation
\[
u_t=6(u+\varepsilon u^2)u_x-u_{xxx},
\]
can in fact be written in Hamiltonian form in two distinct ways. Firstly, we
see
\[
u_t=D_x(3u^2+2\varepsilon u^3-u_{xx})=\mathscr{D}\delta\mathscr{H}_1,
\]
where $\mathscr{D}=D_x$ and
\[
\mathscr{H}_1[u,\varepsilon]=\int(u^3+\frac{\varepsilon}{2}u^4+\frac{u_x^2}{2})\,dx
\]
is an approximate conservation law. Note that $\mathscr{D}$ is certainly skew-adjoint, and  Hamiltonian. The
Poisson bracket is
\[
\{\mathscr{P},\mathscr{L}\}=\int\delta\mathscr{P}\cdot D_x(\delta\mathscr{L})dx
\]
The second Hamiltonian form is
\[
u_t=\Big(4uD_x+2u_x+3\varepsilon(uu_x+u^2D_x)-D_x^3\Big)u=\mathscr{E}\delta\mathscr{H}_0,
\]
in which
\[
\mathscr{H}_0[u,\varepsilon]=\int\frac{1}{2}u^2\,dx
\]
$\mathscr{E}$ is skew-adjoint and  satisfies  the Jacobi identity. Therefore it is Hamiltonian.

In \cite{Nadj1}, we have comprehensively analyzed the problem of approximate symmetries for the Gardner equation.
We have shown that the approximate symmetries of the Gardner equation are given by the following generators:
\[
\mathbf{v}_1=\partial_x,\qquad \mathbf{v}_2=\partial_t,\qquad
\mathbf{v}_3=6t\partial_x+(2\varepsilon u-1)\partial_u,
\]
\[
\mathbf{v}_4=\varepsilon \mathbf{v}_1,\qquad \mathbf{v}_5=\varepsilon \mathbf{v}_2,\qquad \mathbf{v}_6=\varepsilon(6t\partial_x-\partial_u)=\varepsilon \mathbf{v}_3,\qquad
\mathbf{v}_7=\varepsilon(x\partial_x+3t\partial_t-2u\partial_u),
\]
with corresponding characteristics
\[
Q_1=u_x,\qquad Q_2=6(u+\varepsilon u^2)u_x-u_{xxx},\qquad Q_3=6tu_x+1-2\varepsilon u
\]
\[
Q_4=\varepsilon Q_1=\varepsilon u_x,\qquad Q_5=\varepsilon Q_2=\varepsilon(6uu_x-u_{xxx}),\qquad Q_6=\varepsilon Q_3=\varepsilon(6tu_x+1)\]
\[
 Q_7=\varepsilon\Big(2u+xu_x+3t(6uu_x-u_{xxx})\Big),
\]
(up to sign).

For the first Hamiltonian operator $\mathscr{D}=D_x$, there is one independent
nontrivial approximate distinguished functional, the mass $\mathscr{P}_0=\mathcal{M}=\int \,u\,dx$ which is
consequently approximately conserved.

For the above seven characteristics, we have
\begin{equation}
\label{eq:eq36}
Q_i\approx D_x\delta\mathscr{P}_i,\qquad i=1,2,4,5,6,
\end{equation}
with the following approximately conserved functionals:
\[
\mathscr{P}_1=\mathscr{H}_0[u,\varepsilon]=\int\frac{1}{2}u^2\,dx,\qquad\mathscr{P}_2=\mathscr{H}_1[u,\varepsilon]=\int\left(u^3+\frac{\varepsilon}{2}u^4+\frac{1}{2}u_x^2\right)\,dx,\qquad\mathscr{P}_4=\varepsilon\mathscr{P}_1=\int\frac{\varepsilon}{2}u^2\,dx,
\]
\[
\mathscr{P}_5=\varepsilon\mathscr{P}_2=\varepsilon\int\left(u^3+\frac{1}{2}u_x^2\right)\,dx,\qquad\mathscr{P}_6=\varepsilon\int\left(3tu^2+xu\right)\,dx.
\]
For the second Hamiltonian operator $\mathscr{E}=4uD_x+2u_x+3\varepsilon(uu_x+u^2D_x)-D_x^3$,

\begin{equation}
\label{eq:eq37}
Q_i\approx\mathscr{E}\delta\mathscr{\tilde{P}}_i,\qquad i=2,4,5,7,
\end{equation}
the following approximately conserved functionals are the corresponding approximate conservation laws:
\[
\tilde{\mathscr{P}_2}=\mathscr{P}_1=\int\frac{1}{2}u^2\,dx,\qquad\mathscr{\tilde{P}}_4=\frac{\varepsilon}{2}\mathscr{P}_0=\frac{\varepsilon}{2}\int u\,dx,
\]
\[
\mathscr{\tilde{P}}_5=\varepsilon\mathscr{\tilde{P}}_2=\varepsilon\int\frac{1}{2}u^2\,dx,\qquad\mathscr{\tilde{P}}_7=\frac{1}{2}\mathscr{P}_6=\frac{\varepsilon}{2}\int\left(3tu^2+xu\right)\,dx.
\]
 In this case, nothing new is
obtained. Note that the other approximate conservation law $\mathscr{P}_5$ did not arise from one of
the geometrical symmetries. According to Theorem \ref{T:thm3}, however, there is an
approximate Hamiltonian symmetry which gives rises to it, namely $\hat{\mathbf{v}}_{\mathscr{P}_5}$. The characteristic
of this approximate generalized symmetry is
\[
\bar{Q}_5\approx\mathscr{E}\delta\mathscr{P}_5=\mathscr{E}\varepsilon(3u^2-u_{xx})
\approx\varepsilon\Big( u_{xxxxx}-10\,u\,u_{xxx}-20\,u_x\,u_{xx}+30\,u^2\,u_x\Big)
\]
Note that $\bar{Q}_5$ happens to satisfy the Hamiltonian
condition (\ref{eq:eq36}) for $\mathscr{D}$ with the following functional
\[
\bar{\mathscr{P}_5}=\frac{\varepsilon}{2}\int(u_{xx}^2-5\,u^2\,u_{xx}+5\,u^4)\,dx
\]
Consequently, another approximate conservation law is provided for the Gardner equation.
\end{example}

Keeping on this procedure recursively, further
approximate conservation laws could be generated.
But, this procedure will be done in the next section by applying approximate recursion operators.
\section{Approximate Recursion Operators}

\begin{definition}  Let $\Delta$ be a system of perturbed differential equations. An \textit{approximate recursion
operator} for $\Delta$ is a linear  operator  $\mathscr{R}:\mathscr{A}^q\rightarrow\mathscr{A}^q$ in the space of $q$-tuples of
differential functions with the property that whenever $\mathbf{v}_Q$ is an approximate evolutionary
symmetry of $\Delta$, so $\mathbf{v}_{\tilde{Q}}$ is  with $\tilde{Q}\approx\mathscr{R}Q$.
\end{definition}

For nonlinear perturbed systems, there is an analogous criterion for a differential
 operator to be an approximate recursion operator, but to state it we need to introduce the
notion of the (formal) Fr\'{e}chet derivative of a differential function.

\begin{definition}
Let $P[u,\varepsilon] = P(x,u^{(n)},\varepsilon)\in \mathscr{A}^r$ be an $r$-tuple of differential
functions. The \textit{Fr\'{e}chet derivative} of $P$ is the perturbed differential operator $\mathbf{D}_P: \mathscr{A}^q\rightarrow\mathscr{A}^r$
defined so that
\begin{equation}
\label{eq:eq38}
\mathbf{D}_P(Q)=\frac{d}{d\epsilon}\Big|_{\epsilon=0}P[u+\epsilon\, Q[u,\varepsilon]]
\end{equation}
for any $Q \in \mathscr{A}^q$.
\end{definition}
\begin{proposition} If $P \in \mathscr{A}^r$ and $Q \in \mathscr{A}^q$ then
\begin{equation}
\label{eq:eq39}
\mathbf{D}_P(Q)\approx\mathbf{pr\,v}_Q(P).
\end{equation}
\end{proposition}

\begin{theorem}
\label{T:thm4}
Suppose that $\Delta[u,\varepsilon]=0$ be a system of perturbed differential equations. If
 $\mathscr{R}:\mathscr{A}^q\rightarrow\mathscr{A}^q$  is a linear  operator such that
for all solutions $u$ of $\Delta$,
\begin{equation}
\label{eq:eq40}
\mathbf{D}_\Delta\cdot\mathscr{R}\approx\tilde{\mathscr{R}}\cdot\mathbf{D}_\Delta
\end{equation}
where $\tilde{\mathscr{R}}:\mathscr{A}^q\rightarrow\mathscr{A}^q$ is a linear differential  operator,
then $\mathscr{R}$ is an approximate recursion operator for the system.
\end{theorem}

Suppose  that $\Delta[u,\varepsilon]=u_t-K[u,\varepsilon]$ is a perturbed evolution equation. Then
$\mathbf{D}_\Delta=D_t-\mathbf{D}_K$. If $\mathscr{R}$ is an approximate recursion operator, then it is not hard to observe that the
 operator $\tilde{\mathscr{R}}$ in (\ref{eq:eq40}) must be the same as $\mathscr{R}$. Therefore, the condition (\ref{eq:eq40}) in
this case reduces to the commutator condition
\begin{equation}
\label{eq:eq41}
\mathscr{R}_t\approx\left[\mathbf{D}_K,\mathscr{R}\right]
\end{equation}
for an approximate recursion operator of a  perturbed evolution equation.

From (\ref{eq:eq41}),  we can conclude that if $\mathscr{R}$ is an approximate recursion operator, then for all $l\geq 1$
in which $\varepsilon^l\mathscr{R}\neq 0$, $\varepsilon^l\mathscr{R}$ is an approximate recursion operator
\[
(\varepsilon^l\mathscr{R})_t=\varepsilon^l\mathscr{R}_t\approx\varepsilon^l\left[\mathbf{D}_K,\mathscr{R}\right]\approx\left[\mathbf{D}_K,\varepsilon^l\mathscr{R}\right].
\]
In order to illustrate the significance of the above
theorem, we discuss a couple of examples, including the potential Burgers' equation and the Gardner equation.
In the first example, we apply some technical methods, used in Examples 5.8 and 5.30 of \cite{Olver1}.

\begin{example}
Consider the potential Burgers' equation
\[
u_t=u_{xx}+\varepsilon u_x^2.
\]
As mentioned in  \cite{Pak}, approximate symmetries of the potential Burgers' equation are given  by the following twelve vector fields
\[
\mathbf{v}_1=\partial_x,\quad \mathbf{v}_2=\partial_t,\quad
\mathbf{v}_3=x\partial_x+2t\partial_t,\quad\mathbf{v}_4=2t\partial_x-(xu-\varepsilon t\frac{u^2}{2})\partial_u,\quad\mathbf{v}_5=(u-\varepsilon t\frac{u^2}{2})\partial_u,
\]
\[
\mathbf{v}_6=4xt\partial_x+4t^2\partial_t-(x^2+2t)(u-\varepsilon t\frac{u^2}{2})\partial_u,\quad\mathbf{v}_7=\varepsilon \mathbf{v}_1,\quad \mathbf{v}_8=\varepsilon \mathbf{v}_2,\quad \mathbf{v}_9=\varepsilon(x\partial_x+2t\partial_t)=\varepsilon \mathbf{v}_3,
\]
\[
\mathbf{v}_{10}=\varepsilon(2t\partial_x-xu\partial_u)=\varepsilon \mathbf{v}_4,\quad \mathbf{v}_{11}=\varepsilon u\partial_u=\varepsilon \mathbf{v}_5,\quad \mathbf{v}_{12}=\varepsilon(4xt\partial_x+4t^2\partial_t-(x^2+2t)u\partial_u)=\varepsilon \mathbf{v}_6,
\]
plus the infinite family of vector fields
\[
\mathbf{v}_{f,g}=\Big(f(x,t)(1-\varepsilon u)+\varepsilon g(x,t)\Big)\partial_u,
\]
where $f,g$ are arbitrary solutions of the heat equation $u_t=u_{xx}$.

The corresponding characteristics of the first twelve approximate symmetries are
\[
Q_1=u_x,\quad Q_2=u_{xx}+\varepsilon u_x^2,\quad Q_3=xu_x+2t(u_{xx}+\varepsilon u_x^2),\quad Q_4=xu+2tu_x-\varepsilon t\frac{u^2}{2},
\]
\[
Q_5=u-\varepsilon t\frac{u^2}{2},\quad Q_6=(x^2+2t)(u-\varepsilon t\frac{u^2}{2})+4xtu_x+4t^2(u_{xx}+\varepsilon u_x^2),\quad Q_7=\varepsilon Q_1=\varepsilon u_x,
\]
\[
 Q_8=\varepsilon Q_2=\varepsilon u_{xx},\qquad Q_9=\varepsilon Q_3=\varepsilon (xu_x+2tu_{xx}),\qquad
 Q_{10}=\varepsilon Q_4=\varepsilon (xu+2tu_x),
\]
\[
 Q_{11}=\varepsilon Q_5=\varepsilon u,\qquad Q_{12}=\varepsilon Q_6=\varepsilon ((x^2+2t)u+4xtu_x+4t^2u_{xx}),
\]
(up to sign).

Inspection
of $Q_1,\,Q_2,\,Q_7,\,Q_8$ leads us to the conjecture that $\mathscr{R}_1=D_x+\varepsilon u_x$ is an approximate
recursion operator,  since $Q_3=\mathscr{R}_1 Q_1,\,Q_8=\mathscr{R}_1 Q_7,$ etc. To prove this, we note
that the Fr\'{e}chet derivative for the right-hand side of potential Burgers' equation is
\[
D_K=D_x^2+2\varepsilon u_x D_x.
\]
We must verify (\ref{eq:eq41}). The time derivative of the first approximate recursion operator $\mathscr{R}_1$
on the solutions of the potential Burgers' equation is the multiplication operator
\[
(\mathscr{R}_1)_t=(D_x+\varepsilon u_x)_t=\varepsilon u_{xt}=\varepsilon( u_{xxx}+2\varepsilon u_x u_{xx})=\varepsilon u_{xxx}.
\]
On the other hand, the commutator is computed using Leibniz' rule for
differential operators:
\[
[D_K,\mathscr{R}_1]=\varepsilon u_{xxx}.
\]
Comparing these two verifies (\ref{eq:eq41}) and proves that $\mathscr{R}_1$ is an approximate recursion
operator for the potential Burgers' equation.

There is thus an infinite hierarchy of approximate symmetries, with characteristics
$\mathscr{R}_1^kQ_1,\, k = 0, 1, 2,\ldots$ For example, the next characteristic after $Q_{12}$ in the
sequence is
\[
\mathscr{R}_1Q_{12}=\varepsilon\Big((x^2+6t)u_x+2x\,(u+2tu_{xx})+4t^2u_{xxx}\Big).
\]
To obtain the characteristics depending on $x$ and $t$, we require a second approximate
recursion operator, which by inspection, we guess to be
\[
\mathscr{R}_2=t\mathscr{R}_1+\frac{x}{2}.
\]
Using the fact that $\mathscr{R}_1$ satisfies (\ref{eq:eq41}), we readily find
\[
(\mathscr{R}_2)_t=t(\mathscr{R}_1)_t+\mathscr{R}_1=t[D_K,\mathscr{R}_1]+\mathscr{R}_1,
\]
whereas
\[
[D_K,\mathscr{R}_2] = t[D_K,\mathscr{R}_1]+[D_x^2+2\varepsilon u_x D_x
,\frac{1}{2}x]
\]
\[
\hspace{3.2cm}=t[D_K,\mathscr{R}_1]+(D_x+\varepsilon u_x)=t[D_K,\mathscr{R}_1]+\mathscr{R}_1,
\]
proving that $\mathscr{R}_2$ is also an approximate recursion operator. There is thus a doubly
infinite hierarchy of approximate generalized symmetries of potential Burgers' equation, with
characteristics $\mathscr{R}_2^l\mathscr{R}_1^kQ_1,\, k,l\geq 0$. For instance, $Q_2=\mathscr{R}_1Q_1$,
$Q_3=2\mathscr{R}_2\mathscr{R}_1Q_1$
and so on.
\end{example}

\begin{example}
Consider the Gardner equation, which was shown
to have two Hamiltonian structures  with
 \[
 \mathscr{D}=D_x,\qquad\mathscr{E}=4uD_x+2u_x+3\varepsilon(uu_x+u^2D_x)-D_x^3.
 \]
Hence, the operator connecting our hierarchy of approximate Hamiltonian symmetries is
\[
\mathscr{R}=\mathscr{E}\cdot\mathscr{D}^{-1}=4u+3\varepsilon u^2+(2+3\varepsilon u)u_xD_x^{-1}-D_x^2.
\]
Therefore,  our
results on approximate bi-Hamiltonian systems will provide ready-made proofs of the
existence of infinitely many approximate conservation laws and approximate symmetries for the
Gardner equation.

Note
that the Fr\'{e}chet derivative for the right-hand side of Gardner's equation
is
\[
\mathbf{D}_K=6(1+2\varepsilon u)u_x+6(u+\varepsilon u^2)D_x-D_x^3,
\]
and
\[
\mathscr{R}_t=(4+6\varepsilon u)u_t+(2u_{xt}+3\varepsilon u_tu_x+3\varepsilon u u_{xt})D_x^{-1}\hspace{4.3cm}
\]
\[
=12uu_x(2+5\varepsilon u)-(4+6\varepsilon u)u_{xxx}\hspace{5.6cm}
\]
\[
+\Big(6uu_{xx}(2+5\varepsilon u)+12 u_x^2(1+5\varepsilon u)-u_{xxxx}(2+3\varepsilon u)-3\varepsilon u_xu_{xxx}\Big)\,D_x^{-1}.
\]
\end{example}

\begin{theorem}
Let $\bar{Q}_0=\varepsilon u_x$. For each $k\geq 0$, the differential polynomial $\bar{Q}_k=\mathscr{R}^k\bar{Q}_0$ is a total $x$-derivative, $\bar{Q}_k = D_xR_k$, and hence we can recursively define $\bar{Q}_{k+1}=\mathscr{R}\bar{Q}_k$.
 Each $\bar{Q}_k$ is the characteristic of an approximate symmetry of the Gardner equation.
\end{theorem}
\begin{proof}
To prove this theorem, we apply the similar method applied in theorem 5.31 of  \cite{Olver1}.

We proceed by induction on $k$, so suppose that $\bar{Q}_k=\mathscr{R}^k\bar{Q}_0$
 for some $R_k\in\mathscr{A}$. From the form of the approximate recursion operator,
 \[
 \bar{Q}_{k+1}=\varepsilon\Big(4u\bar{Q}_k+2u_xD_x^{-1}\bar{Q}_k-D_x^2\bar{Q}_k\Big)=\varepsilon \,D_x \Big(2uD_x^{-1}\bar{Q}_k+2D_x^{-1}(u\bar{Q}_k)-D_x\bar{Q}_k\Big)
 \]
If we can prove that for some differential polynomial $S_k\in\mathscr{A}$,
$u\bar{Q}_k = D_xS_k$,
we will indeed have proved that $\bar{Q}_{k+1}=D_xR_{k+1}$, where $R_{k+1}$ is the above expression in
brackets. Consequently, the induction step will be completed.

To prove this fact, note  that the formal adjoint of the approximate recursion
operator $\varepsilon\mathscr{R}$ is
\[
\varepsilon\mathscr{R}^*=\varepsilon(4u-2D_x^{-1}\cdot u_x-D_x^2)=D_x^{-1}\,\varepsilon\mathscr{R}\,D_x.
\]
We apply this in order to integrate the expression $u\bar{Q}_k$, by parts, so
\[
\bar{Q}_k=u\mathscr{R}^k[\varepsilon u_x]=u_x\cdot(\varepsilon\mathscr{R}^*)^k[u]+D_xA_k
\]
for some differential function $A_k\in\mathscr{A}$. On the other hand, using a further
integration by parts, for some  $B_k\in\mathscr{A}$ the following identity holds:
\[
u_x\cdot(\varepsilon\mathscr{R}^*)^k[u]=u_x\cdot D_x^{-1}\,\varepsilon\mathscr{R}[u_x]=u_x\cdot D_x^{-1}
\bar{Q}_k=-u\bar{Q}_k+D_xB_k
\]
 Substituting into the previous identity, we conclude
\[
u\bar{Q}_k = D_xS_k,\qquad where\qquad S_k= \frac{1}{2}(A_k + B_k),
\]
which proves our claim.
\end{proof}

\begin{definition} A pair of skew-adjoint $q \,\times\, q$ matrix of differential operators $\mathscr{D}$
and $\mathscr{E}$  is said to form an \textit{approximately Hamiltonian pair} if every linear combination $a\mathscr{D}  + b\mathscr{E}$,
 $a,\,b\in\mathbb{R}$, is an approximate Hamiltonian operator. A system of perturbed evolution equations is an
\textit{approximate bi-Hamiltonian system} if it can be written in the form
\begin{equation}
\label{eq:eq42}
\frac{\partial u}{\partial t}=K_1[u,\varepsilon]\approx\mathscr{D}\delta\mathscr{H}_1\approx\mathscr{E}\delta\mathscr{H}_0
\end{equation}
 where $\mathscr{D}$, $\mathscr{E}$ form
an approximately Hamiltonian pair, and $\mathscr{H}_0$ and $\mathscr{H}_1$ are
appropriate Hamiltonian functionals.
\end{definition}

\begin{lemma} If $\mathscr{D}$, $\mathscr{E}$  are skew-adjoint operators, then they form an approximately Hamiltonian
pair if and only if $\mathscr{D}$, $\mathscr{E}$ and $\mathscr{D}+\mathscr{E}$ are all approximate Hamiltonian operators.
\end{lemma}
\begin{corollary}
Let $\mathscr{D}$ and $\mathscr{E}$  be Hamiltonian differential operators. Then $\mathscr{D}$, $\mathscr{E}$
form an approximately Hamiltonian pair if and only if
\begin{equation}
\label{eq:eq43}
\mathbf{pr\,v}_{\mathscr{D}\theta}(\Theta_\mathscr{E})+\mathbf{pr\,v}_{\mathscr{E}\theta}(\Theta_\mathscr{D})=o(\varepsilon^p),
\end{equation}
where
\[
\Theta_\mathscr{D}=\frac{1}{2}\int\{\theta\wedge\mathscr{D}\theta \}\,dx,\qquad \Theta_\mathscr{E}=\frac{1}{2}\int\{\theta\wedge\mathscr{E}\theta \}\,dx
\]
are the functional bi-vectors representing the respective Poisson brackets.

\end{corollary}
\begin{example}
Consider the approximate Hamiltonian operators  $\mathscr{D}$, $\mathscr{E}$ associated with the
Gardner equation. We have
\[
\mathbf{pr\,v}_{\mathscr{D}\theta}=\sum_{\alpha,J}D_J\left(\mathscr{D}\theta\right)_\alpha\frac{\partial}{\partial u_J^\alpha}=\sum_{\alpha,J}D_J\left(\sum_{\beta=1}^q\mathscr{D}_{\alpha\beta}\theta^\beta\right)\frac{\partial}{\partial u_J^\alpha}
\]
in the case of the second approximate Hamiltonian
operator for the Gardner equation, we have
\[
\mathbf{pr\,v}_{\mathscr{E}\theta}(u)=\mathscr{E}\theta,\qquad\mathbf{pr\,v}_{\mathscr{E}\theta}(u^2)=2u\mathscr{E}\theta,
\]
\[
\mathbf{pr\,v}_{\mathscr{E}\theta}(\Theta_\mathscr{D})=\mathbf{pr\,v}_{\mathscr{E}\theta}\int\frac{1}{2}\{\theta\wedge\theta_x \}\,dx=o(\varepsilon^p)
\]
trivially, by the properties of the wedge product, it is deduced that:
\[
\quad\mathbf{pr\,v}_{\mathscr{D}\theta}(\Theta_\mathscr{E})=\mathbf{pr\,v}_{\mathscr{D}\theta}\int\{(2u+\frac{3\varepsilon }{2}u^2)\theta\wedge\theta_x+\frac{1}{2}\theta_x\wedge\theta\}\,dx\]
\[
\approx\int\{(2+3\varepsilon u)\theta_x\wedge\theta\wedge\theta_x\}=o(\varepsilon^p)
\]
 Thus $\mathscr{D}$ and $\mathscr{E}$ form an approximately Hamiltonian
pair.
\end{example}
\begin{definition}
A differential  operator  $\mathscr{D}:\mathscr{A}^r\rightarrow\mathscr{A}^s$ is \textit{approximately degenerate} if there is a
nonzero differential  operator  $\mathscr{\tilde{D}}:\mathscr{A}^s\rightarrow\mathscr{A}$ such that  $\mathscr{\tilde{D}}\cdot\mathscr{D}\equiv o(\varepsilon^p)$
\end{definition}

Now, according to \cite{Olver1},  we are in a situation to state the main theorem on approximate bi-Hamiltonian
systems.

\begin{theorem}
\label{T:thm5}
Let
\[
u_t=K_1[u,\varepsilon]\approx\mathscr{D}\delta\mathscr{H}_1\approx\mathscr{E}\delta\mathscr{H}_0
\]
be an approximate bi-Hamiltonian system of perturbed evolution equations. Assume that the  operator $\mathscr{D}$
of the approximately Hamiltonian pair is approximate nondegenerate. Let $\mathscr{R} = \mathscr{E}\cdot\mathscr{D}^{-1}$ be the
corresponding approximate recursion operator, and let $K_0\approx\mathscr{D}\delta\mathscr{H}_0$. Assume that for each $n = 1,2,\ldots$
we can recursively define
\[
K_n\approx\mathscr{R}K_{n-1},\qquad n\geq 1,
\]
meaning that for each $n$, $K_{n-1}$ lies in the image of $\mathscr{D}$. Then there exists a
sequence of functionals $\mathscr{H}_0,\,\mathscr{H}_1,\,\mathscr{H}_2,\ldots $ such that
\begin{itemize}
\item[(i)] for each $n\geq 1$  the perturbed evolution equation
\begin{equation}
\label{eq:eq44}
u_t=K_n[u,\varepsilon]\approx\mathscr{D}\delta\mathscr{H}_n\approx\mathscr{E}\delta\mathscr{H}_{n-1}
\end{equation}
is an approximate bi-Hamiltonian system;
\item[(ii)] the corresponding approximate evolutionary vector fields $\mathbf{v}_n=\mathbf{v}_{K_n}$ all mutually
commute:
\[
[\mathbf{v}_n,\mathbf{v}_m]=o(\varepsilon^p),\qquad n,m\geq 0;
\]
\item[(iii)] the approximate Hamiltonian functionals $\mathscr{H}_n$ are all in involution with respect to either
Poisson bracket:
\begin{equation}
\label{eq:eq45}
\left\{\mathscr{H}_n,\mathscr{H}_n\right\}_{\mathscr{D}}=o(\varepsilon^p)=\left\{\mathscr{H}_n,\mathscr{H}_n\right\}_{\mathscr{E}},\qquad n,m\geq 0,
\end{equation}
and hence provide an infinite collection of approximate conservation laws for each of
the approximate bi-Hamiltonian systems (\ref{eq:eq42}).
\end{itemize}
\end{theorem}

We have seen that given an approximate bi-Hamiltonian system, the operator  $\mathscr{R}=\mathscr{E}\cdot\mathscr{D}^{-1}$,
when applied successively to the initial equation $K_0= \mathscr{D}\delta\mathscr{H}_0$, produces an
infinite sequence of approximate generalized symmetries of the original system (subject to
the technical assumptions contained in Theorem \ref{T:thm5}). It is still not clear that
$\mathscr{R}$ is a true approximate recursion operator for the system, in the sense that whenever $\mathbf{v}_Q$ is
an approximate generalized symmetry, so is $\mathbf{v}_{\mathscr{R}Q}$. So far,  we only know it for approximate symmetries with
$Q = K_n$ for some $n$. In order to establish this more general result, we need a
formula for the infinitesimal change of the approximate Hamiltonian operator itself under
a Hamiltonian flow.
\begin{lemma} Let $u_t = K \approx  \mathscr{D}\delta\mathscr{H}$ be an approximate Hamiltonian system of perturbed evolution
equations with corresponding vector field $\mathbf{v}_K=\mathbf{\hat{v}}_{\mathscr{H}}$. Then
\[
\mathbf{pr\,\hat{v}}_{\mathscr{H}}(\mathscr{D})\approx\mathbf{D}_K\cdot\mathscr{D}+\mathscr{D}\cdot\mathbf{D}_K^*
\]
\end{lemma}

\begin{theorem}
\label{T:thm6}
 Let $u_t = K \approx  \mathscr{D}\delta\mathscr{H}_1\approx \mathscr{E}\delta\mathscr{H}_0$ be an approximate bi-Hamiltonian system of perturbed
evolution equations. Then the operators $\mathscr{R}_l=\varepsilon^l\mathscr{E}\cdot\mathscr{D}^{-1},\,0\leq l\leq p,$ are approximate  recursion operators for
the system.
\end{theorem}
Judging from $\mathscr{R}_l^p=o(\varepsilon^p)$, when $l\neq 0$, this type of  approximate
recursion operators have less significance than $\mathscr{R}_0$.

\begin{example}
The approximate recursion operators of the  Gardner equation are
\[
\mathscr{R}_0=\mathscr{E}\cdot\mathscr{D}^{-1}=4u+2u_xD_x^{-1}+3\varepsilon(uu_xD_x^{-1}+u^2)-D_x^2,\qquad
\mathscr{R}_1=\varepsilon(4u+2u_xD_x^{-1}-D_x^2)
\]
and we can apply $\mathscr{R}_0$ to the right-hand side of the Gardner equation to obtain the
approximate symmetries. The first step in this recursion is the flow
\[
u_t\approx\mathscr{E}\delta\mathscr{H}_1\approx\mathscr{D}\delta\mathscr{H}_2
\approx u_{xxxxx}-10\,u\,u_{xxx}-20\,u_x\,u_{xx}+30\,u^2\,u_x+\varepsilon
\Big(55u^3u_x-39uu_xu_{xx}-9u^2u_{xxx}-12u_x^3 \Big),
\]
which is not approximately total derivative, so we can not re-apply the approximate recursion operator to get
a meaningful approximate generalized symmetry.

But if we  set
\[
\bar{K}_1[u,\varepsilon]=Q_5=\varepsilon K_1[u,\varepsilon]=\varepsilon(6uu_x-u_{xxx}),\qquad\bar{\mathscr{H}_0}=\tilde{\mathscr{P}_5}=\varepsilon\mathscr{H}_0,\qquad\bar{\mathscr{H}_1}=\mathscr{P}_5=\varepsilon\mathscr{H}_1,
\]
then we can apply $\mathscr{R}_0$ successively to  $\bar{K}_1$  in order to obtain the
approximate symmetries. The first phase become
\[
u_t\approx\mathscr{E}\delta\bar{\mathscr{H}_1}\approx\mathscr{D}\delta\bar{\mathscr{H}_2}
\approx\mathscr{R}_0\bar{K}_1\approx\varepsilon( u_{xxxxx}-10\,u\,u_{xxx}-20\,u_x\,u_{xx}+30\,u^2\,u_x)
\]
in which
\[
\bar{\mathscr{H}_2}=\bar{\mathscr{P}_5}=\frac{\varepsilon}{2}\int(u_{xx}^2-5\,u^2\,u_{xx}+5\,u^4)\,dx
\]
is another approximate conservation law.

Now, for $\bar{K}_2=\mathscr{R}_0\bar{K}_1$ we have
\[
u_t\approx\mathscr{E}\delta\bar{\mathscr{H}_2}\approx\mathscr{D}\delta\bar{\mathscr{H}_3}
\approx\mathscr{R}_0\bar{K}_2
\approx\varepsilon( -u_{xxxxxxx}+14\,uu_{xxxxx}+42u_xu_{xxxx})
\]
\[
\,+70\,\varepsilon(u_{xx}u_{xxx}-u^2u_{xxx}+2u^3u_x-4uu_xu_{xx}-u_x^3)
\]
where\[
\bar{\mathscr{H}_3}=7\varepsilon\int(\frac{u_{xxx}^2}{14}+uu_{xx}^2+5u^2\,u_x^2+u^5)\,dx
\]
is a further approximate conservation law.
\end{example}

\section{Concluding Remarks}
Sometimes, differential equations appearing in mathematical modelings are written with terms
involving a small parameter which is known as the perturbed term. Taking into account the
instability of the Lie point symmetries with respect to perturbation of coefficients of differential
equations, the approximate (perturbed) symmetries for such equations are obtained.
Different methods for computing the approximate symmetries of a system of differential equations are available
in the literature \cite{Baikov1,Baikov2,Fush}.\\
The approximate symmetry method proposed by Fushchich and Shtelen \cite{Fush} is based on a perturbation of dependent variables.
This method has so many advantages such as producing more approximate group-invariant solutions, consistence with the perturbation theory, solving singular perturbation problems \cite{Pak,Wilst1} and close relationship  with approximate homotopy symmetry method \cite{zhang1}.
But despite above mentioned benefits, this procedure converts a perturbed evolution equation to an equivalent perturbed evolutionary system.
In his case, obtaining the corresponding Hamiltonian formulation will be hard. Due to increase of the
dimensions of Hamiltonian operators $\mathscr{D},\,\mathscr{E}$, computation of the
approximate recursion operator $\mathscr{R}=\mathscr{E}\cdot\mathscr{D}^{-1}$ is difficult.\\
Since prolongation and Fr\'{e}chet derivative of vector fields are linear, both of the
approximate symmetry methods can be extended on the Hamiltonian structures. But due to the significance of
vector fields in Hamiltonian and bi-Hamiltonian systems, the approximate symmetry method proposed by Baikov, Gazizov
and Ibragimov \cite{Baikov1,Baikov2} seems to be more consistent.
\section{Acknowledgements}
It is a pleasure to thank the anonymous referees for their constructive suggestions and helpful comments
which have improved the presentation of the paper. The authors wish to express their sincere gratitude to
Fatemeh Ahangari for her useful advise and suggestions.

\end{document}